\documentclass[10pt,draft]{amsart}

\usepackage[cp1251]{inputenc}
\usepackage[T2A]{fontenc}
\usepackage{amsmath}
\usepackage{amsfonts}
\usepackage{amssymb}
\usepackage{amsthm}
\usepackage[ukrainian,russian,english]{babel}
\theoremstyle{plain}
\newtheorem{theorem}{Theorem}

\newtheorem{proposition}[theorem]{Proposition}

\theoremstyle{definition}

\theoremstyle{remark}

\theoremstyle{plain}
\newtheorem*{theorem*}{Theorem}
\newtheorem*{lemma*}{Lemma}
\newtheorem*{proposition*}{Proposition}
\newtheorem*{statement*}{Statement}
\newtheorem*{corollary*}{Corollary}
\theoremstyle{definition}
\newtheorem*{definition*}{Definition}
\theoremstyle{remark}
\newtheorem*{notation*}{Notation}
\newtheorem*{remark*}{Remark}
\newtheorem*{example*}{Example}
\begin{document}
\title[Spectral gaps of the one-dimensional Schr\"{o}dinger operators]
{Spectral gaps of the one-dimensional Schr\"{o}dinger operators with singular periodic potentials}

\author{Vladimir Mikhailets}
\address{Institute of Mathematics \\
  National Academy of Science of Ukraine \\
  Tereshchenkivska str., 3 \\
  Kyiv-4 \\
  Ukraine \\
  01601}
\email{mikhailets@imath.kiev.ua}

\author{Volodymyr Molyboga}
\address{Institute of Mathematics \\
  National Academy of Science of Ukraine \\
  Tereshchenkivska str., 3 \\
  Kyiv-4 \\
  Ukraine \\
  01601}
\email{molyboga@imath.kiev.ua}

\dedicatory{In the memory of A.Ya. Povzner.}
\begin{abstract}
The behaviour of the lengths of spectral gaps $\{\gamma_{n}(q)\}_{n=1}^{\infty}$ of the Hill-Schr\"{o}dinger operators
\begin{equation*}
  S(q)u=-u''+q(x)u,\quad u\in \mathrm{Dom}\left(S(q)\right)
\end{equation*}
with real-valued 1-periodic distributional potentials $q(x)\in H_{1\mbox{-}per}^{-1}(\mathbb{R})$ is studied. We show that
they exhibit the same behaviour as the  Fourier coefficients $\{\widehat{q}(n)\}_{n=-\infty}^{\infty}$ of the potentials $q(x)$
with respect to the weighted sequence spaces $h^{s,\varphi}$, $s>-1$, $\varphi\in \mathrm{SV}$. The case $q(x)\in L_{1\mbox{-}per}^{2}(\mathbb{R})$, $s\in \mathbb{Z}_{+}$, $\varphi\equiv 1$ corresponds to the Marchenko-Ostrovskii Theorem.
\end{abstract}

\maketitle
\section{Introduction}\label{sec_IntMR}
The Hill-Schr\"{o}dinger operators
\begin{equation*}
  S(q)u:=-u''+q(x)u,\quad u\in \mathrm{Dom}\left(S(q)\right)
\end{equation*}
with real-valued 1-periodic distributional potentials $q(x)\in H_{1\mbox{-}per}^{-1}(\mathbb{R})$ are well defined on the
Hilbert space $L^{2}(\mathbb{R})$ in the following \textit{equivalent} basic ways \cite{MiMl6}:
\begin{itemize}
  \item as form-sum operators;
  \item as quasi-differential operators;
  \item as limits of operators with smooth 1-periodic potentials in the norm resolvent sense.
\end{itemize}
The operators $S(q)$ are lower semibounded and self-adjoint on the Hilbert space $L^{2}(\mathbb{R})$. Their spectra are absolutely continuous and have a band and gap structure as in the classical case of  $L_{1\mbox{-}per}^{2}(\mathbb{R})$-potentials \cite{HrMk, Krt, DjMt3, MiMl6}.

The object of our investigation is the behaviour of the lengths of spectral gaps.  Under the assumption
\begin{equation}\label{eq_12_IntMR}
  q(x)=\sum_{k\in \mathbb{Z}}\widehat{q}(k)e^{i k 2\pi x}\in
  H_{1\mbox{-}per}^{-1+}(\mathbb{R},\mathbb{R}),
\end{equation}
that is
\begin{equation*}
    \sum_{k\in \mathbb{Z}}(1+2|k|)^{2s}|\widehat{q}(k)|^{2}<\infty\quad \forall s>-1,
    \quad\text{and}\quad \mathrm{Im}\,q(x)=0,
\end{equation*}
we will prove many terms asymptotic estimates for the lengths
$\{\gamma_{n}(q)\}_{n=1}^{\infty}$ and midpoints
$\{\tau_{n}(q)\}_{n=1}^{\infty}$ of spectral gaps of the
Hill-Schr\"{o}dinger operators $S(q)$ (Theorem \ref{th_12_IntMR}). These
estimates enable us to establish relationship between the rate of
\textit{decreasing}/\textit{increasing} of the lengths of spectral gaps
and the \textit{regularity} of the singular potentials (Theorem
\ref{th_18_IntMR}  and Theorem~\ref{th_19_IntMR}).


It is well known that if the potentials
\begin{equation*}\label{eq_14_IntMR}
  q(x)=\sum_{k\in \mathbb{Z}}\widehat{q}(k)e^{i k 2\pi x}\in L_{1\mbox{-}per}^{2}(\mathbb{R},\mathbb{R}),
\end{equation*}
i.e., when
\begin{equation*}
 \sum_{k\in \mathbb{Z}}|\widehat{q}(k)|^{2}<\infty,\quad\text{and}\quad \mathrm{Im}\,q(x)=0,
\end{equation*}
then the Hill-Schr\"{o}dinger operators $S(q)$ are lower semibounded and self-adjoint operators on the Hilbert space
$L^{2}(\mathbb{R})$ with absolutely continuous spectra which have a zone structure \cite{DnSch2, ReSi4}.

Spectra $\mathrm{spec}\,(S(q))$ are defined by the location of the endpoints
$\{\lambda_{0}(q),\lambda_{n}^{\pm}(q)\}_{n=1}^{\infty}$ of spectral gaps which satisfy the following
inequalities:
\begin{equation*}
  -\infty<\lambda_{0}(q)<\lambda_{1}^{-}(q)\leq\lambda_{1}^{+}(q)<
  \lambda_{2}^{-}(q)\leq\lambda_{2}^{+}(q)<\cdots\,.
\end{equation*}
Moreover, for even/odd numbers $n\in \mathbb{Z}_{+}$ the endpoints of spectral gaps are eigenvalues of the
periodic/semiperiodic problems on the interval $[0,1]$:
\begin{align*}
   S_{\pm}(q)u & :=-u''+q(x)u=\lambda u, \\
   \mathrm{Dom}(S_{\pm}(q)) & :=\left\{u\in H^{2}[0,1]\left|\, u^{(j)}(0)=\pm\, u^{(j)}(1),\, j=0,1\right.\right\}\equiv H_{\pm}^{2}[0,1].
\end{align*}

Spectral bands (stability or tied zones),
\begin{equation*}
  \mathcal{B}_{0}(q):=[\lambda_{0}(q),\lambda_{1}^{-}(q)],\qquad
  \mathcal{B}_{n}(q):=[\lambda_{n}^{+}(q),\lambda_{n+1}^{-}(q)],\quad n\in
  \mathbb{N},
\end{equation*}
are characterized as a locus of those real $\lambda\in \mathbb{R}$ for which all solutions of the equation $S(q)
u=\lambda u$ are bounded. On the other hand, spectral gaps (instability or forbidden zones),
\begin{equation*}
  \mathcal{G}_{0}(q):=(-\infty,\lambda_{0}(q)),\qquad
  \mathcal{G}_{n}(q):=(\lambda_{n}^{-}(q),\lambda_{n}^{+}(q)),\quad n\in
  \mathbb{N},
\end{equation*}
are a locus of those real $\lambda\in \mathbb{R}$ for which any nontrivial solution of the equation $S(q) u=\lambda u$
is unbounded.

Due to Marchenko and Ostrovskii \cite{MrOs} the endpoints of spectral gaps of the Hill-Schr\"{o}dinger operators $S(q)$
satisfy the asymptotic estimates
\begin{equation}\label{eq_18_IntMR}
  \lambda_{n}^{\pm}(q)=n^{2}\pi^{2}+\widehat{q}(0)\pm \left|\widehat{q}(n)\right|+h^{1}(n),\quad n\rightarrow\infty.
\end{equation}
As a consequence, for the lengths of spectral gaps,
\begin{equation*}
  \gamma_{n}(q):=\lambda_{n}^{+}-\lambda_{n}^{-},\quad n\in \mathbb{N},
\end{equation*}
the following asymptotic formulae are fulfilled,
\begin{equation}\label{eq_20_IntMR}
  \gamma_{n}(q)=2\left|\widehat{q}(n)\right|+h^{1}(n),\quad n\rightarrow\infty.
\end{equation}

Hochstadt \cite{Hchs} ($\Rightarrow$) and Marchenko, Ostrovskii \cite{MrOs}, McKean, Trubowitz \cite{McKTr} ($\Leftarrow$) proved that the potential $q(x)$ is an infinitely differentiable function if and only if the lengths of spectral gaps $\{\gamma_{n}(q)\}_{n=1}^{\infty}$ decrease faster than arbitrary power of $1/n$:
\begin{equation*}\label{eq_21.3_IntMR}
  q(x)\in C_{1\mbox{-}per}^{\infty}(\mathbb{R},\mathbb{R})\Leftrightarrow
  \gamma_{n}(q)=O(n^{-k}),\; n\rightarrow\infty\quad \forall  k\in \mathbb{Z}_{+}.
\end{equation*}

Marchenko and Ostrovskii \cite{MrOs} discovered that
\begin{equation}\label{eq_22_IntMR}
  q(x)\in H_{1\mbox{-}per}^{k}(\mathbb{R},\mathbb{R})\Leftrightarrow
  \{\gamma_{n}(q)\}_{n=1}^{\infty}\in h^{k},\quad k\in \mathbb{Z}_{+}.
\end{equation}

The relationship \eqref{eq_22_IntMR} was extended by Kappeler, Mityagin \cite{KpMt} ($\Rightarrow$) and Djakov, Mityagin
\cite{DjMt1} ($\Leftarrow$) (see also the survey \cite{DjMt2} and the references therein) on the case of symmetric,
monotonic, submultiplicative and strictly subexponential weights $\Omega=\left\{\Omega(n)\right\}_{n\in \mathbb{Z}}$:
\begin{equation*}
  q(x)\in H_{1\mbox{-}per}^{\Omega}(\mathbb{R},\mathbb{R})\Leftrightarrow \{\gamma_{n}(q)\}_{n=1}^{\infty}\in h^{\Omega}.
\end{equation*}
P\"{o}schel \cite{Psch} proved the latter statement in a quite different way.

Here and throughout the remainder of the paper we use the complex Hilbert spaces $H_{1\mbox{-}per}^{w}(\mathbb{R})$ (as well as $H_{\pm}^{w}[0,1]$) of 1-periodic functions and distributions defined by means of their Fourier coefficients:
\begin{align*}
  f(x) & =\sum_{k\in \mathbb{Z}}\widehat{f}(k)e^{i k 2\pi x}\in H_{1\mbox{-}per}^{w}(\mathbb{R}) \Leftrightarrow \left\{\widehat{f}(k)\right\}_{k\in \mathbb{Z}}\in h^{w}, \\
h^{w} & =\left\{a=\left\{a(k)\right\}_{k\in \mathbb{Z}}\,\left|\,\|a\|_{h^{w}}=\left(\sum_{k\in \mathbb{Z}}w^{2}(k)|a(k)|^{2}\right)^{1/2}<\infty\right.\right\}.
\end{align*}
Basically we use the power weights
\begin{equation*}
 w_{s}=\left\{w_{s}(k)\right\}_{k\in \mathbb{Z}}:\qquad w_{s}(k):=(1+2|k|)^{s},\quad s\in \mathbb{R}.
\end{equation*}
The corresponding spaces we denote as
\begin{equation*}
  H_{1\mbox{-}per}^{w_{s}}(\mathbb{R})\equiv H_{1\mbox{-}per}^{s}(\mathbb{R}),\;
  H_{\pm}^{w_{s}}[0,1]\equiv H_{\pm}^{s}[0,1],\quad\text{and}\quad h^{w_{s}}\equiv h^{s},\; s\in \mathbb{R}.
\end{equation*}
For more details, see Appendix.


\section{Main results}\label{ssec_DsPt_IntMR}
As we already remarked, under the assumption \eqref{eq_12_IntMR} the Hill-Schr\"{o}dinger operators $S(q)$ are lower semibounded and
self-adjoint on the Hilbert space $L^{2}(\mathbb{R})$. Their spectra are absolutely continuous and have a classical zone structure \cite{HrMk, Krt, DjMt3, MiMl6, MiSb}.

Using the results of the papers \cite{KpMh, Mhr}, the Isospectral Theorem \ref{th_30_Prf} and \cite[Theorem C]{MiMl6} we
prove uniform many terms asymptotic estimates for the lengths of spectral gaps $\{\gamma_{n}(q)\}_{n=1}^{\infty}$ and their
midpoints $\{\tau_{n}(q)\}_{n=1}^{\infty}$,
\begin{equation*}
  \tau_{n}(q):=\frac{\lambda_{n}^{+}(q)+\lambda_{n}^{-}(q)}{2},\quad n\in \mathbb{N}.
\end{equation*}

\begin{theorem}[\cite{MiMl3, Mlb2}]\label{th_12_IntMR}
Let $q(x)\in H_{1\mbox{-}per}^{-\alpha}(\mathbb{R},\mathbb{R})$, $\alpha\in [0,1)$. Then for any $\varepsilon>0$ uniformly on the bounded sets of distributions $q(x)$ in the corresponding Sobolev spaces $H_{1\mbox{-}per}^{-\alpha}(\mathbb{R})$ the lengths
$\{\gamma_{n}(q)\}_{n=1}^{\infty}$ and midpoints $\{\tau_{n}(q)\}_{n=1}^{\infty}$ of spectral gaps of the Hill-Schr\"{o}dinger
operators $S(q)$ for $n\geq n_{0}\left(\|q\|_{H_{1\mbox{-}per}^{-\alpha}(\mathbb{R})}\right)$ satisfy the following asymptotic
formulae:
\begin{align}
  & \gamma_{n}(q)=2\left|\widehat{q}(n)\right|+h^{1-2\alpha-\varepsilon}(n), \label{eq_24_IntMR} \\
  & \tau_{n}(q)=n^{2}\pi^{2}+\widehat{q}(0)+h^{1-2\alpha-\varepsilon}(n). \label{eq_26_IntMR}
\end{align}
\end{theorem}

\begin{corollary*}[\cite{MiMl3, Mlb2}]\label{cr_16_IntMR}
Let $q(x)\in H_{1\mbox{-}per}^{-\alpha}(\mathbb{R},\mathbb{R})$ with $\alpha\in [0,1)$. Then for any $\varepsilon>0$ uniformly by $q(x)$ for the endpoints of spectral gaps of the Hill-Schr\"{o}dinger operators $S(q)$ the following asymptotic estimates are
fulfilled:
\begin{equation*}\label{eq_28_IntMR}
  \lambda_{n}^{\pm}(q)=n^{2}\pi^{2}+\widehat{q}(0)\pm\left|\widehat{q}(n)\right|+h^{1-2\alpha-\varepsilon}(n).
\end{equation*}
\end{corollary*}

Now, we can describe a bilateral relationship between the rate of
decreasing/increasing of the lengths of spectral gaps
$\{\gamma_{n}(q)\}_{n=1}^{\infty}$ and the regularity of the
potentials $q(x)$ in the refined scale.

Let
\begin{equation*}
 w_{s,\varphi}=\left\{w_{s,\varphi}(k)\right\}_{k\in \mathbb{Z}}:\qquad w_{s,\varphi}(k):=(1+2|k|)^{s}\,
\varphi(|k|),\quad s\in \mathbb{R},\; \varphi\in \mathrm{SV},
\end{equation*}
where $\varphi$ is a slowly varying on $+\infty$ in a sense of Karamata function \cite{Snt}. It means that it is  a
positive, measurable on $[A,\infty)$, $A>0$ function obeying the condition
\begin{equation*}
 \lim_{t\rightarrow +\infty}\frac{\varphi(\lambda t)}{\varphi(t)}=1,\quad \lambda>0.
\end{equation*}
For example,
\begin{equation*}
 \varphi(t)=(\log t)^{r_{1}}(\log\log t)^{r_{2}}\ldots (\log\ldots\log t)^{r_{k}}\in \mathrm{SV},\quad
\{r_{1},\ldots,r_{k}\}\subset \mathbb{R},\; k\in \mathbb{N}.
\end{equation*}

The H\"{o}rmander spaces
\begin{equation*}
  H_{1\mbox{-}per}^{w_{s,\varphi}}(\mathbb{R})\equiv H_{1\mbox{-}per}^{s,\varphi}(\mathbb{R})\simeq H^{s,\varphi}(\mathbb{S}),\quad \mathbb{S}:=\mathbb{R}/2\pi\mathbb{Z},
\end{equation*}
and the weighted sequence spaces
\begin{equation*}
 h^{w_{s,\varphi}}\equiv h^{s,\varphi}
\end{equation*}
form the refined scales:
\begin{align}
 H_{1\mbox{-}per}^{s+\varepsilon}(\mathbb{R})&\hookrightarrow H_{1\mbox{-}per}^{s,\varphi}(\mathbb{R})\hookrightarrow H_{1\mbox{-}per}^{s-\varepsilon}(\mathbb{R}), \label{eq_27.1_IntMR} \\
h^{s+\varepsilon}&\hookrightarrow h^{s,\varphi}\hookrightarrow h^{s-\varepsilon}, \hspace{90pt} s\in \mathbb{R},\,
\varepsilon>0,\,\varphi\in \mathrm{SV},
\label{eq_27.2_IntMR}
\end{align}
which in a general situation were studied by Mikhailets and Murach \cite{MiMr}.

The following statements show that the sequence $\{\gamma_{n}(q)\}_{n=1}^{\infty}$ has the same behaviour as the Fourier
coefficients $\{\widehat{q}(n)\}_{n=-\infty}^{\infty}$ with respect to the refined scale $\{h^{s,\varphi}\}_{s\in
\mathbb{R},\varphi\in \mathrm{SV}}$.
\begin{theorem}\label{th_18_IntMR}
Let $q(x)\in H_{1\mbox{-}per}^{-1+}(\mathbb{R},\mathbb{R})$. Then
\begin{equation*}\label{eq_30_IntMR}
  q(x)\in H_{1\mbox{-}per}^{s,\varphi}\left(\mathbb{R},\mathbb{R}\right)\Leftrightarrow
  \{\gamma_{n}(q)\}_{n=1}^{\infty}\in h^{s,\varphi},\quad s\in (-1,0], \varphi\in \mathrm{SV}.
\end{equation*}
\end{theorem}

Note that the H\"{o}rmander spaces $H_{1\mbox{-}per}^{s,\varphi}(\mathbb{R})$  with $\varphi\equiv 1$ coincide with the Sobolev spaces,
\begin{equation*}
 H_{1\mbox{-}per}^{s,1}(\mathbb{R})\equiv H_{1\mbox{-}per}^{s}(\mathbb{R}),\quad\mbox{and}\quad h^{s,1}\equiv h^{s},\quad s\in \mathbb{R}.
\end{equation*}

\begin{corollary*}[\cite{MiMl3, Mlb2}]
 Let $q(x)\in H_{1\mbox{-}per}^{-1+}(\mathbb{R},\mathbb{R})$, then
\begin{equation}\label{eq_31_IntMR}
  q(x)\in H_{1\mbox{-}per}^{s}\left(\mathbb{R},\mathbb{R}\right)\Leftrightarrow
  \{\gamma_{n}(q)\}_{n=1}^{\infty}\in h^{s},\quad s\in (-1,0].
\end{equation}
\end{corollary*}

Theorem \ref{th_18_IntMR} together with \cite[Theorem 1.2]{KpMt}, and the properties \eqref{eq_27.1_IntMR} and \eqref{eq_27.2_IntMR},
involve the following extension of the Marchenko-Ostrovskii Theorem \eqref{eq_22_IntMR}.
\begin{theorem}\label{th_19_IntMR}
 Let $q(x)\in H_{1\mbox{-}per}^{-1+}(\mathbb{R},\mathbb{R})$. Then
\begin{equation*}
  q(x)\in H_{1\mbox{-}per}^{s,\varphi}\left(\mathbb{R},\mathbb{R}\right)\Leftrightarrow
  \{\gamma_{n}(q)\}_{n=1}^{\infty}\in h^{s,\varphi},\quad s\in (-1,\infty), \varphi\in \mathrm{SV}.
\end{equation*}
In particular,
\begin{equation*}
  q(x)\in H_{1\mbox{-}per}^{s}\left(\mathbb{R},\mathbb{R}\right)\Leftrightarrow
  \{\gamma_{n}(q)\}_{n=1}^{\infty}\in h^{s},\quad s\in (-1,\infty).
\end{equation*}
\end{theorem}

\begin{remark*}\label{rm_20_IntMR}
In the preprint \cite{DjMt3} the authors announced without a proof the more general statement:
\begin{equation*}
  q(x)\in H_{1\mbox{-}per}^{\widehat{\Omega}}(\mathbb{R},\mathbb{R})\Leftrightarrow
  \{\gamma_{n}(q)\}_{n=1}^{\infty}\in h^{\widehat{\Omega}},\quad
  \widehat{\Omega}=\left\{\frac{\Omega(n)}{1+2|n|}\right\}_{n\in \mathbb{Z}},
\end{equation*}
where the weights $\Omega=\{\Omega(n)\}_{n\in \mathbb{Z}}$ are supposed to be symmetric, monotonic, submultiplicative and
strictly subexponential ones. This result contains the limiting case
\begin{equation*}
 q(x)\in H_{1\mbox{-}per}^{-1}\left(\mathbb{R},\mathbb{R}\right)\setminus H_{1\mbox{-}per}^{-1+}\left(\mathbb{R},\mathbb{R}\right).
\end{equation*}
An implication
\begin{equation*}
 q(x)\in H_{1\mbox{-}per}^{-1}\left(\mathbb{R},\mathbb{R}\right)\Rightarrow
  \{\gamma_{n}(q)\}_{n=1}^{\infty}\in h^{-1}
\end{equation*}
was proved in the paper \cite{Krt}.
\end{remark*}


\section{Proofs}\label{sec_Prf}
Spectra of the Hill-Schr\"{o}dinger operators $S(q)$, $q(x)\in
H_{1\mbox{-}per}^{-1}\left(\mathbb{R},\mathbb{R}\right)$ are defined by the endpoints
$\{\lambda_{0}(q),\lambda_{n}^{\pm}(q)\}_{n=1}^{\infty}$ of spectral gaps. The endpoints as in the case of
$L_{1\mbox{-}per}^{2}(\mathbb{R})$-potentials satisfy the inequalities:
\begin{equation*}
  -\infty<\lambda_{0}(q)<\lambda_{1}^{-}(q)\leq\lambda_{1}^{+}(q)<
  \lambda_{2}^{-}(q)\leq\lambda_{2}^{+}(q)<\cdots\,.
\end{equation*}
For even/odd numbers $n\in \mathbb{Z}_{+}$ they are eigenvalues of the periodic/semiperiodic problems on the interval
$[0,1]$ \cite[Theorem C]{MiMl6},
\begin{equation*}
  S_{\pm}(q)u=\lambda u.
\end{equation*}
The operators
\begin{align*}
   & \hspace{15pt} S_{\pm}u\equiv S_{\pm}(q)u:=D_{\pm}^{2}u+q(x)u, \hspace{125pt}\mbox{} \\
   & \bullet\; D_{\pm}^{2}:=-d^{2}/dx^{2},\; \mathrm{Dom}\,(D_{\pm}^{2})=H_{\pm}^{2}[0,1]; \\
   & \bullet\; q(x)=\sum_{k\in \mathbb{Z}}\widehat{q}(k)\,e^{i\,k 2\pi x}\in H^{-1}_{+}\left([0,1],\mathbb{R}\right); \\
   & \bullet\; \mathrm{Dom}\left(S_{\pm}(q)\right)=\left\{u\in H_{\pm}^{1}[0,1]\,\left|\, D_{\pm}^{2}u+q(x)u
   \in L^{2}(0,1)\right.\right\},
\end{align*}
are well defined on the Hilbert space $L^{2}(0,1)$ as lower semibounded, self-adjoint form-sum operators, and they
have the pure discrete spectra
\begin{equation*}
  \mathrm{spec}\left(S_{\pm}(q)\right)=\left\{\lambda_{0}[S_{+}(q)],\; \lambda_{2n-1}^{\pm}[S_{-}(q)],\;
  \lambda_{2n}^{\pm}[S_{+}(q)]\right\}_{n=1}^{\infty}.
\end{equation*}

In the papers \cite{MiMl3, Mlb2, MiMl4, MiMl5} the authors meticulously investigated the more general
periodic/semiperiodic form-sum operators
\begin{equation*}
   S_{m,\pm}(V):=D_{\pm}^{2m}\dotplus V(x),\quad V(x)\in H_{+}^{-m}[0,1],\; m\in \mathbb{N},
\end{equation*}
on the Hilbert space $L^{2}(0,1)$.

So, we need to find precise asymptotic estimates for the operators $S_{\pm}(q)$ eigenvalues. It is quite difficult problem
as the form-sum operators $S_{\pm}(q)$ are not convenient for investigation. We also cannot apply
approach developed by Savchuk and Shkalikov (see the survey \cite{SvSh} and the references therein) considering the
operators $S_{\pm}(q)$ as quasi-differential ones as the periodic/semiperiodic boundary conditions are not strongly
regular by Birkhoff. Therefore we propose an alternative approach which is based on isospectral transformation of the problem.

Kappeler and M\"{o}hr \cite{KpMh, Mhr} investigated the second order differential operators $L_{\pm}(q)$, $q(x)\in
H_{+}^{-1}\left([0,1],\mathbb{R}\right)$ (in general, with complex-valued potentials) defined on the \textit{negative}
Sobolev spaces $H_{\pm}^{-1}[0,1]$,
\begin{equation*}
  L_{\pm}\equiv L_{\pm}(q):=D_{\pm}^{2}+q(x),\quad \mathrm{Dom}\,(L_{\pm}(q))=H_{\pm}^{1}[0,1].
\end{equation*}
They established that the operators $L_{\pm}(q)$ with $q(x)\in H_{+}^{-\alpha}\left([0,1],\mathbb{R}\right)$,
$\alpha\in [0,1)$ have the real-valued discrete spectra
\begin{equation*}
  \mathrm{spec}\left(L_{\pm}(q)\right)=\left\{\lambda_{0}[L_{+}(q)],\; \lambda_{2n-1}^{\pm}[L_{-}(q)],\;
  \lambda_{2n}^{\pm}[L_{+}(q)]\right\}_{n=1}^{\infty}
\end{equation*}
such that
\begin{equation*}
  \left|\lambda_{n}^{\pm}[L_{\pm}(q)]-n^{2}\pi^{2}-\widehat{q}(0)\right|\leq C n^{\alpha},
  \quad n\geq n_{0}\left(\|q\|_{H_{+}^{-\alpha}[0,1]}\right).
\end{equation*}
More precisely, for the values
\begin{align*}
  \gamma_{n}[L_{\pm}(q)] & :=\lambda_{n}^{+}[L_{\pm}(q)]-\lambda_{n}^{-}[L_{\pm}(q)],\quad n\in \mathbb{N}, \\
  \tau_{n}[L_{\pm}(q)] & :=\frac{\lambda_{n}^{+}[L_{\pm}(q)]+\lambda_{n}^{-}[L_{\pm}(q)]}{2},\quad n\in \mathbb{N}
\end{align*}
they proved the following result.
\begin{proposition}[Kappeler, M\"{o}hr \cite{KpMh, Mhr}]\label{pr_26_Prf}
Let $q(x)\in H_{+}^{-\alpha}\left([0,1],\mathbb{R}\right)$, and $\alpha\in [0,1)$. Then for any $\varepsilon>0$ uniformly on the bounded sets
of distributions $q(x)$ in the Sobolev spaces $H_{+}^{-\alpha}[0,1]$ for the operators $L_{\pm}(q)$ values
$\left\{\gamma_{n}[L_{\pm}(q)]\right\}_{n=1}^{\infty}$ and $\left\{\tau_{n}[L_{\pm}(q)]\right\}_{n=1}^{\infty}$ for
$n\geq n_{0}\left(\|q\|_{H_{+}^{-\alpha}[0,1]}\right)$ the following asymptotic estimates are fulfilled:
\begin{align*}
  i) \hspace{15pt} & \left\{\min_{\pm}\left|\gamma_{n}[L_{\pm}(q)]\pm 2\sqrt{\left(\widehat{q}+
  \omega\left)(-n)\right(\widehat{q}+\omega\right)(n)}\right|\right\}_{n\in \mathbb{N}}
  \in h^{1-2\alpha-\varepsilon}, \hspace{70pt}\mbox{} \\
  ii) \hspace{15pt} & \tau_{n}[L_{\pm}(q)]=n^{2}\pi^{2}+\widehat{q}(0)+h^{1-2\alpha-\varepsilon}(n),
\end{align*}
where the convolution
\begin{equation*}
  \left\{\omega(n)\right\}_{n\in \mathbb{Z}}\equiv \left\{\frac{1}{\pi^{2}}\sum_{k\in \mathbb{Z}\setminus\{\pm n\}}
  \frac{\widehat{q}\,(n-k)\widehat{q}(n+k)}{n^{2}-k^{2}}\right\}_{n\in \mathbb{Z}}\in
  \begin{cases}
    h^{1-\alpha}, &  \alpha\in [0,1/2), \\
    h^{3/2-2\alpha-\delta}, &  \alpha\in [1/2,1)
  \end{cases}
\end{equation*}
with any $\delta>0$ (see the Convolution Lemma \cite{KpMh, Mhr}).
\end{proposition}
\begin{remark*}\label{rm_28_Prf}
In the papers \cite{Mlb1, MiMl1, MiMl2, Mlb2} more general operators
\begin{equation*}
  L_{m,\pm}(V):=D_{\pm}^{2m}+ V(x),\quad V(x)\in H_{+}^{-m}[0,1],\; m\in \mathbb{N}
\end{equation*}
on the spaces $H_{\pm}^{-m}[0,1]$ were studied. In particular, the analogue of Proposition \ref{pr_26_Prf} was proved.
\end{remark*}

The following statement is an essential point of our approach.
\begin{theorem}[Isospectral Theorem \cite{MiMl3, Mlb2}]\label{th_30_Prf}
The operators $S_{\pm}(q)$ and $L_{\pm}(q)$ are isospectral ones:
\begin{equation*}
  \mathrm{spec}\left(S_{\pm}(q)\right)=\mathrm{spec}\left(L_{\pm}(q)\right).
\end{equation*}
\end{theorem}
\begin{proof}
The injections
\begin{equation*}
  \mathrm{spec}\left(S_{\pm}(q)\right)\subset \mathrm{spec}\left(L_{\pm}(q)\right)
\end{equation*}
are obvious since
\begin{equation*}
  S_{\pm}(q)\subset L_{\pm}(q).
\end{equation*}
Let prove the inverse injections
\begin{equation*}
  \mathrm{spec}\left(L_{\pm}(q)\right)\subset
  \mathrm{spec}\left(S_{\pm}(q)\right).
\end{equation*}
Let $\lambda\in \mathrm{spec}\left(L_{\pm}(q)\right)$, and $f$ be a correspondent eigenvector or rootvector. Therefore
\begin{equation*}
  \left(L_{\pm}(q)-\lambda Id\right)f=g,\quad f,g\in
  \mathrm{Dom}\left(L_{\pm}(q)\right)=H_{\pm}^{1}[0,1],
\end{equation*}
where $f$ is an eigenfunction if $g=0$, and a rootvector if $g\neq 0$.

So, we have got
\begin{equation*}
  L_{\pm}(q)f=\lambda Id f+g\in H_{\pm}^{1}[0,1],
\end{equation*}
i.e.,
\begin{equation*}
  L_{\pm}(q)f=D_{\pm}^{2}f+q(x)f\in L^{2}(0,1).
\end{equation*}
Thus we have proved that $f\in \mathrm{Dom}\left(S_{\pm}(q)\right)$. In the case when $f$ is a rootvector ($g\neq
0$) in a similar fashion we show that $g\in \mathrm{Dom}\left(S_{\pm}(q)\right)$ also. Continuing this process until necessary (draw attention that it is finite as the eigenvalue $\lambda$ has a finite algebraic multiplicity) we
obtain that all correspondent to $\lambda$ eigenvectors and rootvectors belong to the operators $S_{\pm}(q)$ domains
$\mathrm{Dom}\left(S_{\pm}(q)\right)$. Consequently we can conclude that
\begin{equation*}
  \lambda\in \mathrm{spec}\left(S_{\pm}(q)\right),
\end{equation*}
and the required injections
\begin{equation*}
  \mathrm{spec}\left(L_{\pm}(q)\right)\subset \mathrm{spec}\left(S_{\pm}(q)\right)
\end{equation*}
have been proved.

The proof is complete.
\end{proof}

Now, Theorem \ref{th_12_IntMR} follows from Proposition \ref{pr_26_Prf}, the Isospectral Theorem \ref{th_30_Prf} and
\cite[Theorem~C]{MiMl6}, since
\begin{align*}
  \widehat{q}(n) & =\overline{\widehat{q}(-n)},\quad  n\in \mathbb{Z}, \\
  \omega(n) & =\overline{\omega(-n)},\quad  n\in \mathbb{Z},
\end{align*}
and as a consequence
\begin{equation*}
  \min_{\pm}\left|\gamma_{n}(q)\pm 2\sqrt{\left(\widehat{q}+\omega\left)(-n)\right(\widehat{q}+
  \omega\right)(n)}\right|=\left|\gamma_{n}(q)-2\left|\left(\widehat{q}+\omega\right)(n)\right|\right|.
\end{equation*}

The proof of Theorem \ref{th_12_IntMR} is complete.

To prove Theorem \ref{th_18_IntMR} we firstly prove its Corollary. The formula \eqref{eq_31_IntMR} follows from \cite[Corollary 0.2 (2.6)]{KpMh}, the Isospectral Theorem \ref{th_30_Prf} and \cite[Theorem C]{MiMl6}. Also it can be proved directly as well as \cite[Corollary~0.2~(2.6)]{KpMh} using the estimates \eqref{eq_24_IntMR}.

Further, to prove Theorem \ref{th_18_IntMR} it is sufficient to apply the asymptotic estimates \eqref{eq_24_IntMR}, the
properties \eqref{eq_27.1_IntMR} and \eqref{eq_27.2_IntMR} of the refined scales, and the formula \eqref{eq_31_IntMR}:
\begin{align*}
 &\mbox{\textbullet}\; & q\in  H_{1\mbox{-}per}^{s,\varphi}\left(\mathbb{R},\mathbb{R}\right)
 &\overset{\eqref{eq_27.1_IntMR}}{\Longrightarrow} q\in
 H_{1\mbox{-}per}^{s-\delta}\left(\mathbb{R},\mathbb{R}\right),\;\delta>0\overset{\eqref{eq_24_IntMR}}{\Longrightarrow}
 \gamma_{n}=2\left|\widehat{q}(n)\right|+h^{1+2(s-\delta)-\varepsilon}(n) \\
&&& \overset{\eqref{eq_27.2_IntMR}}{\Longrightarrow}\gamma_{n}=2\left|\widehat{q}(n)\right|+h^{s,\varphi}(n)\Longrightarrow
  \{\gamma_{n}(q)\}_{n=1}^{\infty}\in h^{s,\varphi}; \\
 &\mbox{\textbullet}\; & \{\gamma_{n}(q)\}_{n=1}^{\infty}\in h^{s,\varphi}&\overset{\eqref{eq_27.2_IntMR}}{\Longrightarrow}
 \{\gamma_{n}\}_{n=1}^{\infty}\in h^{s-\delta},\;\delta>0\overset{\eqref{eq_31_IntMR}}{\Longrightarrow}
 q\in  H_{1\mbox{-}per}^{s-\delta}\left(\mathbb{R},\mathbb{R}\right) \\
&&& \overset{\eqref{eq_24_IntMR}}{\Longrightarrow} \gamma_{n}=2\left|\widehat{q}(n)\right|+h^{1+2(s-\delta)-\varepsilon}(n)
 \overset{\eqref{eq_27.2_IntMR}}{\Longrightarrow}\gamma_{n}=2\left|\widehat{q}(n)\right|+h^{s,\varphi}(n) \\
&&& \Longrightarrow \{\widehat{q}(n)\}_{n\in \mathbb{Z}}\in h^{s,\varphi}(n).
\end{align*}
Remark that due to arbitrary choice of $\delta>0$ and $\varepsilon>0$ we may choose them such that
\begin{equation*}
 1+s-2\delta-\varepsilon>0.
\end{equation*}

The proof of Theorem \ref{th_18_IntMR} is complete.
Now, we are ready to prove Theorem \ref{th_19_IntMR}.

At first, note that from \cite[Theorem 1.2]{KpMt} we get the following asymptotic formulae for the lengths of spectral gaps:
\begin{equation}\label{eq_40_Prf}
  \gamma_{n}(q)=2\left|\widehat{q}(n)\right|+h^{1+s}(n)\quad\text{as}\quad q(x)\in
  H_{1\mbox{-}per}^{s}\left(\mathbb{R},\mathbb{R}\right),\;s\in [0,\infty),
\end{equation}
which for the integer numbers $s\in \mathbb{Z}_{+}$ were proved by Marchenko and Ostrovskii \cite{MrOs}.

Using \eqref{eq_31_IntMR}, \eqref{eq_40_Prf} and \eqref{eq_22_IntMR} it is easy to prove the relationship
\begin{equation}\label{eq_42_Prf}
  q(x)\in H_{1\mbox{-}per}^{s}\left(\mathbb{R},\mathbb{R}\right)\Leftrightarrow
  \{\gamma_{n}(q)\}_{n=1}^{\infty}\in h^{s},\quad s\in (-1,\infty).
\end{equation}

\textit{Sufficiency} of Theorem \ref{th_19_IntMR}. Let $q(x)\in
H_{1\mbox{-}per}^{s,\varphi}\left(\mathbb{R},\mathbb{R}\right)$. If $s\in (-1,0]$ then due to Theorem \ref{th_18_IntMR} we
obtain that $ \{\gamma_{n}(q)\}_{n=1}^{\infty}\in h^{s,\varphi}$. If $s>0$ then
\begin{align*}
  q(x)\in H_{1\mbox{-}per}^{s,\varphi}\left(\mathbb{R},\mathbb{R}\right)&\overset{\eqref{eq_27.1_IntMR}}{\hookrightarrow}
  H_{1\mbox{-}per}^{s-\delta}\left(\mathbb{R},\mathbb{R}\right),\;\delta>0\overset{\eqref{eq_40_Prf}}{\Longrightarrow}
  \gamma_{n}(q)=2\left|\widehat{q}(n)\right|+h^{1+s-\delta}(n) \\
  & \overset{\eqref{eq_27.2_IntMR}}{\Longrightarrow}\gamma_{n}(q)=2\left|\widehat{q}(n)\right|+h^{s,\varphi}(n)
  \Longrightarrow \{\gamma_{n}(q)\}_{n=1}^{\infty}\in h^{s,\varphi}.\hspace{50pt}\mbox{}
\end{align*}

Sufficiency is proved.

\textit{Necessity} of Theorem \ref{th_19_IntMR}. Let suppose that $\{\gamma_{n}(q)\}_{n=1}^{\infty}\in h^{s,\varphi}$. If
$s\in (-1,0]$ then from Theorem \ref{th_18_IntMR} we have that $q(x)\in
H_{1\mbox{-}per}^{s,\varphi}\left(\mathbb{R},\mathbb{R}\right)$. If $s>0$ then
\begin{align*}
 \{\gamma_{n}(q)\}_{n=1}^{\infty}\in h^{s,\varphi}&\overset{\eqref{eq_27.2_IntMR}}{\hookrightarrow}h^{s-\delta},\;\delta>0
 \overset{\eqref{eq_42_Prf}}{\Longrightarrow} q(x)\in H_{1\mbox{-}per}^{s-\delta}\left(\mathbb{R},\mathbb{R}\right)
 \overset{\eqref{eq_40_Prf}}{\Longrightarrow}\gamma_{n}(q)=2\left|\widehat{q}(n)\right|+h^{1+s-\delta}(n) \\
 & \overset{\eqref{eq_27.2_IntMR}}{\Longrightarrow}\gamma_{n}(q)=2\left|\widehat{q}(n)\right|+h^{s,\varphi}(n)
 \Longrightarrow q(x)\in H_{1\mbox{-}per}^{s,\varphi}\left(\mathbb{R},\mathbb{R}\right).
\end{align*}

Necessity is proved.

The proof of Theorem \ref{th_19_IntMR} is complete.
\section{Concluding remarks}
In fact we can prove the following result: if $q(x)\in  H_{1\mbox{-}per}^{-1+}\left(\mathbb{R},\mathbb{R}\right)$
and
\begin{align*}
 (1+2|k|)^{s}& \ll w(k)\ll (1+2|k|)^{1+2s},\quad s\in (-1,0], \\
 (1+2|k|)^{s}& \ll w(k)\ll (1+2|k|)^{1+s},\quad s\in [0,\infty),
\end{align*}
then
\begin{equation*}
 q(x)\in H_{1\mbox{-}per}^{w}\left(\mathbb{R},\mathbb{R}\right)\Leftrightarrow
  \{\gamma_{n}(q)\}_{n=1}^{\infty}\in h^{w}.
\end{equation*}

This result is not covered by the theorems of the preprint
\cite{DjMt3}, because it does not require the weight function to
be monotonic and submultiplicative.
\section*{Appendix}
The complex Sobolev spaces $H_{1\mbox{-}per}^{s}(\mathbb{R})$, $s\in \mathbb{R}$ of 1-periodic func\-tions and dis\-tri\-butions
over the real axis $\mathbb{R}$ are defined by means of their Fourier coefficients,
\begin{align*}
  H_{1\mbox{-}per}^{s}(\mathbb{R}) & :=\left\{f=\sum_{k\in \mathbb{Z}}\widehat{f}(k)e^{i k2\pi
  x}\left|\;\parallel f\parallel_{H_{1\mbox{-}per}^{s}(\mathbb{R})}<\infty\right.\right\}, \\
  \parallel f\parallel_{H_{1\mbox{-}per}^{s}(\mathbb{R})} & :=\left(\sum_{k\in \mathbb{Z}}
  \langle 2k\rangle^{2s}|\widehat{f}(k)|^{2}\right)^{1/2},\quad \langle k\rangle:=1+|k|, \\
  \widehat{f}(k) & :=\langle f,e^{i  k2\pi x}\rangle_{L_{1\mbox{-}per}^{2}(\mathbb{R})},\quad k\in \mathbb{Z}.
\end{align*}
By $\langle \cdot,\cdot\rangle_{L_{1\mbox{-}per}^{2}(\mathbb{R})}$
we denote a sesqui-linear form pairing the dual spaces
$H_{1\mbox{-}per}^{s}(\mathbb{R})$ and
$H_{1\mbox{-}per}^{-s}(\mathbb{R})$ with respect to $L_{1\mbox{-}per}^{2}(\mathbb{R})$, which (the sequi-linear form
$\langle \cdot,\cdot\rangle_{L_{1\mbox{-}per}^{2}(\mathbb{R})}$) is
an extension by continuity of the
$L_{1\mbox{-}per}^{2}(\mathbb{R})$-inner product \cite{Brz, GrGr}:
\begin{equation*}
  \langle f,g\rangle_{L_{1\mbox{-}per}^{2}(\mathbb{R})}:=\int_{0}^{1}f(x)\overline{g(x)}\,dx
  =\sum_{k\in \mathbb{Z}}\widehat{f}(k)\overline{\widehat{g}(k)}\quad
  \forall f,g\in L_{1\mbox{-}per}^{2}(\mathbb{R}).
\end{equation*}
It is useful to notice that
\begin{equation*}
  H_{1\mbox{-}per}^{0}(\mathbb{R})\equiv L_{1\mbox{-}per}^{2}(\mathbb{R}).
\end{equation*}

By $H_{1\mbox{-}per}^{s+}(\mathbb{R})$, $s\in \mathbb{R}$ we denote an inductive limit of the Sobolev spaces
$H_{1\mbox{-}per}^{t}(\mathbb{R})$ with $t>s$,
\begin{equation*}
  H_{1\mbox{-}per}^{s+}\left(\mathbb{R}\right):=\bigcup_{\varepsilon>0}H_{1\mbox{-}per}^{s+\varepsilon}\left(\mathbb{R}\right).
\end{equation*}
It is a topological space with an inductive topology.

In a similar fashion the Sobolev spaces $H_{\pm}^{s}[0,1]$, $s\in \mathbb{R}$ of
1-periodic/1-semiperiodic functions and distributions over the interval $[0,1]$ are defined:
\begin{align*}
  H_{\pm}^{s}[0,1] & :=\left\{f=\sum_{k\in \Gamma_{\pm}}\widehat{f}\left(\frac{k}{2}\right)e^{i k\pi
  x}\left|\;\parallel f\parallel_{H_{\pm}^{s}[0,1]}<\infty\right.\right\}, \\
  \parallel f\parallel_{H_{\pm}^{s}[0,1]}&:=\left(\sum_{k\in \Gamma_{\pm}}
  \langle k\rangle^{2s}\left|\widehat{f}\left(\frac{k}{2}\right)\right|^{2}\right)^{1/2},\quad \langle k\rangle=1+|k|, \\
  \widehat{f}\left(\frac{k}{2}\right) &:=\langle f(x),e^{i k\pi x}\rangle_{\pm}, \quad k\in \Gamma_{\pm}.
\end{align*}
Here
\begin{align*}
  \Gamma_{+}  \equiv 2\mathbb{Z} & :=\left\{ k\in\mathbb{Z}\; \left|\;k\equiv 0\; (\mathrm{mod}\,2)\right.\right\}, \\
  \Gamma_{-}  \equiv 2\mathbb{Z}+1 & :=\left\{ k\in\mathbb{Z}\; \left|\;k\equiv 1\; (\mathrm{mod}\,2)\right.\right\},
\end{align*}
and $\langle\cdot,\cdot\rangle_{\pm}$ are sesqui-linear forms pairing the dual spaces
$H_{\pm}^{s}[0,1]$ and $H_{\pm}^{-s}[0,1]$ with respect to $L^{2}(0,1)$, which (the sesqui-linear forms $\langle\cdot,\cdot\rangle_{\pm}$) are
extensions by continuity of the $L^{2}(0,1)$-inner product \cite{Brz, GrGr}:
\begin{equation*}
  \langle f,g\rangle_{\pm}:=\int_{0}^{1}f(x)\overline{g(x)}\,dx =\sum_{k\in
  \Gamma_{\pm}}\widehat{f}\left(\frac{k}{2}\right)\overline{\widehat{g}\left(\frac{k}{2}\right)}\quad \forall f,g\in L^{2}(0,1).
\end{equation*}
It is obvious that
\begin{equation*}
  H_{+}^{0}[0,1]\equiv H_{-}^{0}[0,1]\equiv L^{2}(0,1).
\end{equation*}

We say that 1-periodic function or distribution $f(x)$ is \textit{real-valued} if $\mathrm{Im}\,f(x)=0$. Let us remind
that
\begin{equation*}
  \mathrm{Re}\,f(x):=\frac{1}{2}(f(x)+\overline{f(x)}),\quad
  \mathrm{Im}\,f(x):=\frac{1}{2i}(f(x)-\overline{f(x)}),
\end{equation*}
(see, for an example, \cite{Vld}). In terms of the Fourier coefficients we have
\begin{equation*}
   \mathrm{Im}\,f(x)=0\Leftrightarrow\widehat{f}(k)=\overline{\widehat{f}(-k)},\quad  k\in \mathbb{Z}.
\end{equation*}
Set
\begin{align*}
  H_{1\mbox{-}per}^{s}(\mathbb{R},\mathbb{R}) & :=\left\{f(x)\in H_{1\mbox{-}per}^{s}(\mathbb{R}) \left|\,\mathrm{Im}\,f(x)=0
  \right.\right\}, \\
  H_{1\mbox{-}per}^{s+}(\mathbb{R},\mathbb{R}) & :=\left\{f(x)\in H_{1\mbox{-}per}^{s+}(\mathbb{R}) \left|\,\mathrm{Im}\,f(x)=0
  \right.\right\}, \\
  H_{\pm}^{s}\left([0,1],\mathbb{R}\right) & :=\left\{f(x)\in H_{\pm}^{s}[0,1] \left|\,\mathrm{Im}\,f(x)=0
  \right.\right\}.
\end{align*}

Also we will need the Hilbert sequence spaces
\begin{equation*}
    h^{s}\equiv h^{s}\left(\mathbb{Z},\mathbb{C}\right),\quad s\in \mathbb{R}
\end{equation*}
of (two-sided) weighted sequences,
\begin{align*}
  h^{s} & :=\left\{a=\left\{a(k)\right\}_{k\in \mathbb{Z}}\,\left|\,\|a\|_{h^{s}}:=
  \left(\sum_{k\in \mathbb{Z}}\langle k\rangle^{2s}|a(k)|^{2}\right)^{1/2}<\infty\right.\right\},\quad \langle
  k\rangle=1+|k|.
\end{align*}
Note that
\begin{equation*}
  h^{0}\equiv l^{2}\left(\mathbb{Z},\mathbb{C}\right),
\end{equation*}
and
\begin{equation*}
  a=\left\{a(k)\right\}_{k\in \mathbb{Z}}\in h^{s},\;s\in \mathbb{R}\qquad\Rightarrow\qquad a(k)=o(|k|^{-s}),
  \quad k\rightarrow \pm\infty.
\end{equation*}


\end{document}